\documentclass{amsart}

\usepackage{graphicx}
\usepackage{tikz-cd}
\usepackage{hyperref}

\newtheorem{theorem}{Theorem}[section]
\newtheorem{lemma}[theorem]{Lemma}

\theoremstyle{definition}
\newtheorem{definition}[theorem]{Definition}
\newtheorem{example}[theorem]{Example}

\theoremstyle{remark}
\newtheorem{remark}[theorem]{Remark}

\numberwithin{equation}{section}

\begin{document}

\title[Permutation formula for $A_n$]{The formula for the permutation of mutation sequences in $A_n$ straight orientation}
\author{Kiyoshi Igusa}
\address{Department of Mathematics, Brandeis University, Waltham, Massachusetts 02453}
\curraddr{Department of Mathematics, Brandeis University, Waltham, Massachusetts 02453}
\email{igusa@brandeis.edu}
\author{Ying Zhou}
\address{Department of Mathematics, Brandeis University, Waltham, Massachusetts 02453}
\curraddr{Department of Mathematics, Brandeis University, Waltham, Massachusetts 02453}
\email{yzhou935@brandeis.edu}




\begin{abstract}
In this paper we state and prove a formula for the  permutations associated to reddening and loop sequences in $A_n$ straight orientation using the picture group. In particular this applies to maximal green sequences in $A_n$ straight orientation. Furthermore we extend the definition and formula of the associated permutation to arbitrary mutation sequences based on our results. We introduce the concept of standard matrices which gives a canonical order on indecomposable components of cluster-tilting objects. Preservation of standardness of $C$-matrices by a combination of a mutation and its associated permutation gives the formula.\\
\end{abstract}

\maketitle

\section{Introduction}
\subsection{Maximal Green Sequences and Reddening Sequences}
\indent A \textit{cluster quiver} is defined as a quiver which is a directed graph with possibly multiple arrows between two vertices but without loops or 2-cycles. The vertex set of a quiver $Q$ is denoted $Q_0$ and its edges set is denoted $Q_1$. For a cluster quiver $Q$ with $n$ vertices the edge set is determined by the $n\times n$ matrix $B$ with entries $b_{ij}$ equal to the number of arrows from $i$ to $j$ minus the number of arrow from $j$ to $i$. This is called the \emph{exchange matrix}.

\begin{definition}
(Keller, 2011) Let $Q$ be a cluster quiver. The \textit{framed quiver} $\hat{Q}$ of $Q$ is obtained from $Q$ by adding a vertex $i'$ and an arrow $i\rightarrow i'$ for every $i\in Q_0$.\cite{Kel11}\\
Let $Q$ be a cluster quiver. The \textit{coframed quiver} $\breve{Q}$ of $Q$ is obtained from $Q$ by adding a vertex $i'$ and an arrow $i'\rightarrow i$ for every $i\in Q$.\cite{BDP13}
\end{definition}

\indent If $Q$ has $n$ vertices, $\hat Q$ and $\breve Q$ have $2n$ vertices. Framed and coframed quivers are both special cases of \textit{ice quivers} which we define here. An \textit{ice quiver} is a quiver $Q$ where a possibly empty set, $F\subseteq Q_0$, consists of vertices that may not mutate.\cite{BDP13} We recall that, for any vertex $k\in Q_0$ the \emph{mutation of $Q$ in the direction of $k$} is the quiver $\mu_k(Q)$ obtained from $Q$ by
\begin{enumerate}
\item reversing all arrows to and from vertex $k$ and
\item adding $b_{ik}|b_{kj}|$ arrows from $i$ to $j$ when $b_{ik},b_{kj}$ have the same sign.
\item removing any 2-cycles created in step (2).\cite{FZ01}
\end{enumerate}

 For an ice quiver $(Q,F)$ we are not allowed to mutate at elements of $F$, so we call them \textit{frozen vertices}. A non-frozen vertex $i$ is \textit{green} if no arrow from a frozen vertex to $i$ exists. Otherwise it is \textit{red}.\cite{Kel11}\\ 
\indent Now we can define green sequences, maximal green sequences and reddening sequences since we have the concepts of ``green" and ``red". Green sequences and maximal green sequences were introduced by Bernhard Keller.\cite{Kel11} A more general form, reddening sequences or green-to-red sequences are used in \cite{BHIT15}\cite{Mul15}.

\begin{definition}
(Keller, 2011) (1) A \textit{green sequence} is a sequence $\mathbf{i}=(i_1, i_2,\cdots, i_N)$ such that for all $1\leq t\leq N$ the vertex $i_t$ is green in the partially mutated quiver $\hat{Q}(\mathbf{i},t)=\mu_{i_{t-1}}\cdots\mu_2\mu_1(\hat{Q})$.\cite{Kel11}\\
(2) A \textit{maximal green sequence} is a green sequence such that $\hat{Q}(\mathbf{i},N)$ does not have any green vertices (and hence can not be extended).\cite{Kel11}\\
(3) A \textit{green-to-red sequence} or a \textit{reddening sequence}, is a sequence $\mathbf{i}=(i_1, i_2,\cdots, i_N)$ that transforms $\hat{Q}$ to a quiver $\hat{Q}(\mathbf{i},N) = \mu_{i_N}\cdots\mu_2\mu_1(\hat{Q})$ such that $\hat{Q}(\mathbf{i},N)$ does not have any green vertices.\cite{Mul15}
\end{definition}

\subsection{The Permutation Associated with a Reddening Sequence}
\indent All reddening sequences have associated permutations. When comparing the quivers obtained from transforming the same framed quiver using two different reddening sequences, it is easy to see that they are just one permutation away from each other: If you do a correct permutation of vertices (that means both rows and columns together) you can transform one such matrix into another. In particular any quiver obtained by using a reddening sequence to transform a framed quiver is one permutation away from the coframed quiver.\\
\indent Here is the formal definition of such a permutation:

\begin{definition}
A \textit{permutation} from an ice quiver $(Q,F)$ to $(Q',F)$ is an isomorphism of quivers $Q\rightarrow Q'$ that preserve $F$.\cite{BDP13}
\end{definition}

\indent We can denote the permutation by elements of the permutation group $S_{|Q|-|F|}$. We have a result from \cite{BDP13} which helps us define the permutation:

\begin{theorem}
(Brustle-Dupont-Perotin \cite{BDP13}) Let $Q$ be a cluster quiver and let $Q'$ be a quiver that is a result of a reddening sequence on $\hat{Q}$, then  $Q'$ equals to a permutation of $\breve{Q}$. In other words, for a reddening sequence $\mathbf{i}=(i_1,\cdots, i_N)$, for some $\rho\in S_n$ we have $\mu_{i_N}\cdots\mu_{i_1}\hat{Q}=\rho\breve{Q}$.
\end{theorem}

\begin{definition}
(Garver-Musiker \cite{GM14}) The \textit{permutation of a reddening sequence} $\mathbf{i}$ is $\rho$ for which $\mu_{i_N}\cdots\mu_{i_1}\hat{Q}=\rho\breve{Q}$.
\end{definition}

At a workshop in Snowbird, Utah in 2014 the question was raised: What is the formula for this permutation? We answer this question for the quiver of type $A_n$ with straight orientation: $1\to 2\to 3\to\cdots\to n$.\\
\indent Here is one of the simplest examples of the concept of the permutation:
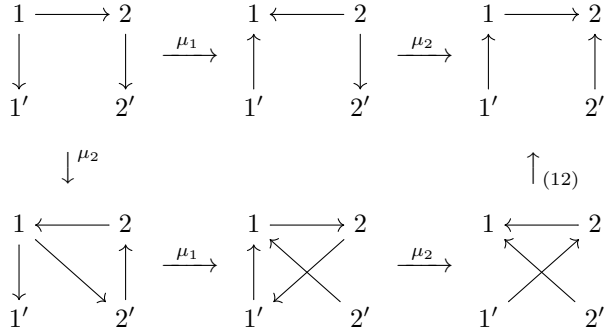
\[
\begin{tikzcd}
1 \arrow{r}\arrow{d} & 2\arrow {d} \\ 
1' & 2'
\end{tikzcd}
\xrightarrow{\ \mu_1\ }
\begin{tikzcd}
1 & 2 \arrow{l}\arrow {d} \\ 
1'\arrow{u} & 2'
\end{tikzcd}
\xrightarrow{\ \mu_2\ }
\begin{tikzcd}
1\arrow{r} & 2 \\ 
1'\arrow{u} & 2' \arrow {u}
\end{tikzcd}
\]
\[
	\quad\big\downarrow^{\mu_2} \qquad\qquad\qquad\qquad\qquad\qquad\quad\quad\qquad \big\uparrow_{(12)} 
\]
\[
\begin{tikzcd}
1 \arrow{dr}\arrow{d} & 2\arrow {l} \\ 
1' & 2'\arrow{u}
\end{tikzcd}
\xrightarrow{\ \mu_1\ }
\begin{tikzcd}
1\arrow{r} & 2 \arrow {dl} \\ 
1'\arrow{u} & 2'\arrow{ul}
\end{tikzcd}
\xrightarrow{\ \mu_2\ }
\begin{tikzcd}
1& 2\arrow{l}  \\ 
1'\arrow{ur} & 2' \arrow {ul}
\end{tikzcd}
\]
It is obvious that the result of $\mu_2\mu_1$ and $\mu_2\mu_1\mu_2$ are not identical, though they can be transformed into each other by a single permutation on vertices.\\
\indent Following \cite{BDP13} we make the $c$-vectors into the rows of the $c$-matrix $C$ and place it to the right of the exchange matrix $B$. For example, the extended exchange matrix for the initial framed quiver above is:
\[
	\tilde B_0=[B_0|C]=\left[
	\begin{array}{cc|cc}
	0 & 1 & 1 & 0\\
	-1 & 0 & 0 & 1
	\end{array}
	\right]
\] The extended exchange matrix for an ice quiver $(Q,F)$ is the $n\times 2n$ matrix $\tilde B=[B|C]$ whose $ij$-entry is the number of arrows in $Q$ from vertex $i$ to vertex $j$ minus the number of arrows from $j$ to $i$ with $i'\in F$ counted as $i+n$. We can also conduct permutation on matrices. In particular permutation of rows of the $c$-matrix can be easily induced by permutation on non-frozen vertices.

\begin{definition}
For a matrix $M=(M_1,\cdots, M_n)^t$ and a permutation $\sigma\in S_n$, if $C=(M_{\sigma^{-1}(1)},\cdots,M_{\sigma^{-1}(n)})^t$, we denote this as $C=\sigma(M)$. Thus $\sigma(M)$ is obtained from $M$ by permuting its rows by $\sigma$.
\end{definition}

\subsection{Picture groups}
\indent We also need to use the concept of the picture groups in order to state the formula below.\\
\indent For a quiver $Q$ of finite type (A,D,E), let $\Phi^+(Q)$ denote the set of all dimension vectors of indecomposable representations of $Q$ over a field $K$. Elements $\beta\in\Phi^+(Q)$ are called (positive) \textit{root}. For each root $\beta$ there is, up to isomorphism, a unique indecomposable modules $M_\beta$ with dimension vector $\beta$. We say that $\beta'$ is a \emph{subroot} of $\beta$ and write $\beta'\subseteq \beta$ if $M_{\beta'}\subseteq M_\beta$. Given two roots $\alpha,\beta$ we use the notation $hom(\alpha,\beta),ext(\alpha,\beta)$ for the dimensions over the ground field of $Hom(M_\alpha,M_\beta)$ and $Ext(M_\alpha,M_\beta)$.

Let $D(\beta)\subseteq\mathbb{R}^n$ be given by
\[
D(\beta)= \{x\in\mathbb{R}^n: \left<x,\beta\right>=0,\ \left<x,\beta'\right>\leq 0\text{ when }\beta'\subseteq\beta\}.\]
Here $\left<\,,\,\right>:\mathbb R^n\to \mathbb R$ denotes the \emph{Euler-Ringel pairing} given by $\left<x,y\right>=x^tEy$ where $E$ is the \emph{Euler matrix} with entries $E_{ij}=hom(e_i,e_j)-ext(e_i,e_j)$. The union of $D(\beta)$ for all these roots divides $\mathbb{R}^n$ into \textit{compartments}. The boundary of each compartment is the union of portions of these $D(\beta)$ which we call \textit{walls}.\cite{IT17}\cite{IOTW4} Sometimes we abuse notation and use the root $\beta$ to mean the wall $D(\beta)$ when the meaning is clear. We also use the notation $+\beta$ to mean the wall $\beta$ is a part of the boundary of a compartment $\mathcal{U}$ and for any point $x\in\mathcal{U}$, $\left<x,\beta\right>\ >0$. Similarly we have the notation $-\beta$. For example $+\beta-\beta'$ means that $\beta$ and $\beta'$ are parts of the boundary of a compartment $\mathcal{U}$ and for any point $x\in\mathcal{U}$, $\left<x,\beta\right>\ >0$ and $\left<x,\beta'\right>\ <0$.\\
\indent In $A_n$, the positive roots are $\beta_{ij}=e_{i+1}+e_{i+2}+\cdots+e_j$ ($0<i<j<n$).

\begin{definition}
\cite{IOTW4},\cite{IT17} The \textit{picture group} of an acyclic quiver of finite type $Q$ is a group $G(Q)=\left<S|R\right>$ with $S$ in bijection with $\Phi^+(Q)$ (the generator corresponding to $\beta$ is $x(\beta)$) and $R$ the set of relations $x(\beta_i)x(\beta_j)=\Pi x(\gamma_k)$ with $\gamma_k$ running over all positive roots which are linear combinations $\gamma_k = a_k\beta_i+b_k\beta_j$ with $a_k/b_k$ increasing (going from 0/1 where $\gamma_1=\beta_j$ to 1/0 where $\gamma_k=\beta_i$) for any pair $(\beta_i,\beta_j)$ such that they are Hom-orthogonal and $ext(\beta_i,\beta_j)=0$.\cite{IT17}
\end{definition}

\indent For a quiver of type $A_n$ we often simplify the notation of $x(\beta_{ij})$ to $x_{ij}$ which we use interchangeably with $x(\beta_{ij})$. The picture group for $A_n$ straight orientation ($1\to\cdots\to n$) is $G(A_n)=\{S|R\}$, $S=\{x_{ij}|0\leq i<j\leq n\}$, $R=\{x_{ij}x_{kl}=x_{kl}x_{ij}|[i,j]\cap[k,l]=\emptyset, [i,j]\text{ or }[k,l],  i,j,k,l \text{ distinct}\}\cup\{x_{jk}x_{ij}=x_{ij}x_{ik}x_{jk}|0\leq i<j<k\leq n\}$.\\
\indent We say that the picture group element $x_{ij}$ is \emph{allowed to act} on $(Q,F)$ if $\beta_{ij}$ is one of the $c$-vectors of $Q$ (one of the rows of $C$). Then the mutation $x_{ij}(Q)$ is defined to be $\mu_k(Q)$ if $\beta_{ij}=c_k$ is the $k$-th $c$ vector of $Q$. Otherwise $x_{ij}(Q)$ is undefined. Similarly, $x_{ij}^{-1}(Q)$ is defined and equal to $\mu_k(Q)$ if $-\beta_{ij}=c_k$ is the $k$-th $c$-vector of $Q$. By the sign coherence property\cite{FST11}\cite{ST12} we know that the exchange matrix $B$ is uniquely determined by the $c$-matrix $C=(c_1,\cdots,c_n)^t$ and the initial exchange matrix $B_0=E^t-E$ by the equation $B=CB_0C^t$\cite{NZ11}. Hence we refer to mutations on the extended exchange matrix $\tilde{B}$ as mutations on $C$. For example $x_{ij}(C)$ will be the right half of the extended exchange matrix $\tilde{B'}=\mu_k(\tilde{B})=\mu_k[CB_0C^t,\ C]$ if the $k$-th row of $C$ is $c_k=\beta_{ij}$.\\
\indent Todorov proved with the first author in \cite{IT17} that there exists a bijection between the set of maximal green sequences and the set $\mathcal{P}(c)$ of positive expressions of the Coxeter element of the picture group for any acyclic valued quiver of finite type. In other words, any positive expression $w\in\mathcal{P}(c)$ is a sequence of allowable green mutations on the initial framed quiver $\hat Q$ giving a maximal green sequence and all maximal green sequences are given uniquely in this way. This applies in particular to $A_n$ straight orientation where the Coxeter element is
\[
	c=\prod_{k=1}^n x_{k,k-1}=x_{n,n-1}\cdots x_{21}x_{10}.
\]
As indicated, multiplication order is from right to left.\\
\indent Finally, we observe that the actions of picture group elements and permutations commute since picture group elements perform mutations independently of how nonfrozen vertices are numbered.

\section{The formula of the associated permutation in $A_n$ straight orientation}
\subsection{Loop Sequences}
\indent Now let's define a new term, namely \textit{loop sequences} which is essential to the discussion about the permutation:

\begin{definition} 
A loop sequence $w$ is a sequence of mutations $\mu_{i_k}\cdots\mu_{i_1}$ on an ice quiver $(Q,F)$ such that $\mu_{i_k}\cdots\mu_{i_1}(Q) = \rho(Q)$ for some permutation $\rho$.
\end{definition}

\indent Using the c-vector theorem from \cite{ST12} we can see that the only nontrivial effect loop sequences can have on an extended matrix is a permutation. Hence loop sequences are equivalent to permutations on an extended matrix.

\begin{definition}
For any loop sequence $w$ the permutation $\rho$ such that $w(\tilde{B})=\rho(\tilde{B})$ is defined as \textit{the associated permutation of the loop sequence $w$}.
\end{definition}

\indent In essence for all acyclic quivers, green-to-red sequences in general and maximal green sequences in particular do not have a natural definition of the permutation: The traditional one in essence is the permutation of an associated loop sequence: Take the reddening sequence and then do mutations at sinks only, go over all non-frozen vertices and return to the origin which constitutes the loop sequence we need.

\subsection{The formula of the permutation}
The formula is given as:

\begin{theorem}
In $A_n$ straight orientation, the permutation associated with a picture group element $\prod_{k}x_{i_kj_k}^{\delta_k}$ which is in correspondence with a reddening or loop sequence acting on a $c$-matrix with associated permutation $\sigma$ is \[
\rho(\prod_{k}x_{i_kj_k}^{\delta_k}) = \sigma(\prod_{k}(i_k+1,j_k))^{-1}\sigma^{-1}.\]
Here $\delta_k\in\{+,-\}$ and multiplication is right to left.
\end{theorem}

\indent This formula works for any maximal green, reddening, and loop sequences. It also extended the definition of an associated permutation to the set of arbitrary finite sequences of mutations in $A_n$ straight orientation. One interesting property of $A_n$ is that the associated permutation of a mutation only depends on the $c$-vector but not which cluster-tilting object on which the mutation is conducted. The associated permutation in the general case seems far less regular.

\subsection{Forbidden pairs of $c$-vectors}
\indent Since we use picture groups and related structures to prove the theorem, we need to examine what kind of pairs of walls can not exist in any compartment. This is given by the following theorem of Speyer and Thomas.

\begin{theorem}\cite{ST12}\label{forbidden pairs of walls} The $c$-vectors of any cluster tilting object for any acyclic quiver are (positive or negative) real Schur roots. For any two such roots of the same sign, the corresponding indecomposable modules are Hom-orthogonal. For any pair of $c$-vectors of opposite sign, say $\alpha,-\beta$, $hom(\alpha,\beta)=0=ext(\alpha,\beta)$.
\end{theorem}

\subsection{Proof of the formula}
\indent The basic idea in proving the theorem is below:\\
\indent Let's first prove it in the simple case, namely when the sequence of mutations begin at a standard $C$-matrix which includes the case of reddening sequences and loop sequences that start from the framed or coframed quivers. In this case the formula is simplified to $\rho(\prod_{k}x_{i_kj_k}^{\delta_k}) = \prod_{k}(i_k+1,j_k)^{-1}$ where inverse means reversing the order of multiplication. Since picture group elements commute with permutations, what we want to prove can be reduced to  $\rho(\prod_{k}(i_k+1,j_k)x_{i_kj_k}^{\delta_k})=id$. This property can further be reduced to proving that for all $k$, $(i_k+1,j_k)x_{i_kj_k}^{\delta_k}$ keeps the c-vectors in a ``standard order,'' a concept which we now define.

\begin{definition}
An $n\times n$ matrix $M\in M_n(\mathbb{Z})$ is \textit{standard} if the following holds:\\
1. The diagonal entries are all nonzero.\\
2. All positive entries can only exist on the diagonal or above. and all negative entries can only exist on the diagonal or below.\\
3. All rows are in the form $\pm\beta_{ij}$.
\end{definition}

\begin{remark}\label{rem: standard matrices}It is easy to see that all rows of the form $-\beta_{ij}$ has to be the $(i+1)$-th row and all rows of the form $\beta_{ij}$ has to be the $j$-th row since all other positions violate either Axiom 1 or 2.
\end{remark}

\begin{example}
\indent Here are several examples:\\\\
$\begin{bmatrix}
1 & 1\\
0 & -1\\
\end{bmatrix}$ and
$\begin{bmatrix}
1 & 0\\
-1 & -1\\
\end{bmatrix}$ are standard matrices because all three axioms hold.\\\\
$\begin{bmatrix}
0 & -1\\
-1 & 0\\
\end{bmatrix}$ and
$\begin{bmatrix}
-1 & -1\\
1 & 0\\
\end{bmatrix}$ are not standard matrices since axioms 1 and 2 are violated.\\
\end{example}
\indent Now let's prove one more lemma, namely if any permutation of rows is performed on a standard matrix, the result is not standard unless the permutation is trivial.

\begin{lemma}
The only permutation of rows on a standard matrix that results in a standard matrix is the trivial permutation.
\end{lemma}

\begin{proof} This follows from Remark \ref{rem: standard matrices}.
\end{proof}

\indent Since any nontrivial permutation on a standard matrix results in a non-standard one, we can define what it means to be the associated permutation of a $c$-matrix in $A_n$ straight orientation as long as it is the result of a permutation on a standard matrix:

\begin{definition}
For a matrix $C$ if there exists a standard matrix $M$ such that $C=\rho(M)$, we define the \textit{associated permutation of $C$} as $\rho$.
\end{definition}

\indent The technical lemma to be proven that can almost immediately lead to the theorem is stated below:

\begin{lemma}
In $A_n$ straight orientation, if $x_{i_kj_k}$ (resp. $x_{i_kj_k}^{-1}$) is an allowable mutation of a standard $c$-matrix $C_{k-1}$ then $C_k=(i_k+1,j_k)x_{i_kj_k}C_{k-1}$ (resp. $C_k=(i_k+1,j_k)x_{i_kj_k}^{-1}C_{k-1}$) is also a standard matrix.
\end{lemma}

\begin{proof}
\indent We will only prove in the green case $x_{i_kj_k}$ since the red case is almost identical to the green one. In this proof $i_k$ is simplified as $i$ and $j_k$ is simplified as $j$. We note that, since our quiver is simply laced of finite type, all entries of the exchange matrix $B$ are $0,1$ or $-1$. So, the mutation $\mu_j$ transformed each $c$-vector $c_s\neq c_j$ to either $c_s'=c_s$ or $c_s'=c_s+c_j$.\\
\indent Case 1: If $j - i  = 1$. Here we have a simple root and the associated permutation $(i+1,j)$ is the identity. Hence the proof reduces to $x_{ij}$ transforms a standard matrix to another standard one. $x_{ij}$ merely flips the $j$-th row from $e_j$ to $-e_j$ and may lengthen some $\beta_{\ell i}$ to $\beta_{\ell j}$ for $l<i$ and shorten some $-\beta_{i\ell}$ to $-\beta_{j\ell}$ for $l>j$ without changing which row they are in. Since no other operation will send $c_s$ to a root $c_s'=c_s+c_j$, the resulting matrix is still standard.\\
\indent Case 2: If $j - i > 1$. Here the associated permutation is the transposition $(i+1,j)$. Now let's discuss what $(i+1,j) x_{ij}$ actually does on each row:\\
\indent a) $l\leq i$. Due to Theorem \ref{forbidden pairs of walls} all c-vectors have to be $\pm\beta_{ab}$ for some $0\leq a<b\leq n$. So the only change that can ever happen is that $\beta_{\ell i}$ may be lengthen to $\beta_{\ell j}$. Changing the $C$-matrix is this way does not violate standardness.\\
\indent b) $\ell=i+1$. By assumption, $c_{i+1}=\beta_{ij}$ which changes to $-\beta_{ij}$. Since we are transposing the $i+1$st and $j$-th rows, this change also does not violate standardness of the $C$-matrix.\\
\indent c) $i + 1 < \ell < j$. Here, the only way that the $c$-vector $c_\ell$ can change to $c_\ell'=c_\ell+\beta_{ij}$ and still remain a root is if $c_\ell=-\beta_{i\ell}$. But this violates Theorem \ref{forbidden pairs of walls} since $hom(\beta_{ij},\beta_{i\ell})\neq0$. So, $c_\ell'=c_\ell$ for all $\ell$ in this range.\\
\indent d) $\ell = j$. We notice several facts:\\
\indent (i).The $c$-vector $c_j$ can not be positive since, in that case, $c_j=\beta_{j-1,m})$ which is not Hom-orthogonal to $c_{i+1}=\beta_{ij}$. Hence we can assume that the $j$-th row is $-\beta_{mj}$ for some $m<j$.\\
\indent (ii). $m>i$. Otherwise $hom(\beta_{ij},\beta_{mj})\neq0$ violating Theorem \ref{forbidden pairs of walls}.\\
\indent (iii). Also it is impossible for $c_j=-\beta_{mj}$ to remain itself after doing $x_{ij}$ since, otherwise, $c_j'$ and $c_i'=-\beta_{ij}$ have the same sign but are not Hom-orthogonal.\\
\indent Hence $c_j=-\beta_{mj}$, $m>i$ and $c_j'=\beta_{ij}-\beta_{mj}=\beta_{im}$ which, when transposed to the $i+1$-st row keeps the matrix standard.\\
\indent e) $\ell>j$. $-\beta_{i\ell}$ can be shortened to $-\beta_{j\ell}$. In all other cases, $c_\ell+\beta_{ij}$ is not a root. So all rows are good and the matrix remains standard.\\
\indent Using similar methods we can see that $(i+1, j)x_{ij}^{-1}$ transforms a standard matrix into another one.
\end{proof}

\indent Let's first prove the formula for reddening sequences and in particular maximal green sequences.

\begin{proof}
Since $(i_k+1,j_k)x_{i_kj_k}$ and $(i_k+1,j_k)x_{i_kj_k}^{-1}$ transform standard matrices to standard matrices, the result of transforming the initial standard matrix, $I_n$ by $\prod_{k}(i_k+1,j_k)x_{i_kj_k}^{\delta_k}$ is standard. Since the sequence is a reddening sequence, the result of this transformation has to be a permutation of rows of $-I_n$, which has to be $-I_n$ itself. Hence $\rho(\prod_{k}(i_k+1,j_k)x_{i_kj_k}^{\delta_k})=id$, the formula is correct.
\end{proof}

\indent Then let's prove a weaker version of the formula, namely when the initial $c$-matrix, $C$, is standard and the sequence is a loop sequence.

\begin{proof}
The result of transforming the initial standard matrix, $C$ by $\prod_{k}(i_k+1,j_k)x_{i_kj_k}^{\delta_k}$ is standard. Since the sequence is a loop sequence, the result of this transformation has to be a permutation of rows of $C$, which has to be $C$ itself. Hence $\rho(\prod_{k}(i_k+1,j_k)x_{i_kj_k}^{\delta_k})=id$, the formula is correct.
\end{proof}

\indent Finally let's prove the formula for all loop sequences.

\begin{proof}
Since any $c$-matrix, $C$ in $A_n$ straight orientation is reachable from the framed quiver, we can use the formula to calculate the associated permutation of $C$ which we denote as $\sigma$. Thus $M=\sigma^{-1}C$ is standard. Since picture group elements commute with permutations we can replace $\prod_{k}x_{i_kj_k}^{\delta_k}$ by $\sigma(\prod_{k}x_{i_kj_k}^{\delta_k})\sigma^{-1}$. Then $\sigma(\prod_{k}x_{i_kj_k}^{\delta_k})\sigma^{-1}(C) = \sigma(\prod_{k}x_{i_kj_k}^{\delta_k})(M) = \sigma(\prod_{k}(i_k+1,j_k))^{-1}(M)=\sigma(\prod_{k}(i_k+1,j_k))^{-1}\sigma^{-1}(C)$.
\end{proof}

\subsection{The formula of associated permutation for any mutation sequences}
\indent Due to the theorem we can extend the definition of associated permutations to any arbitrary mutation sequence in $A_n$ straight orientation which reduces to the existing definitions of the associated permutation of reddening and loop sequences due to the theorem above.

\begin{definition}
In $A_n$ straight orientation, the \textit{associated permutation of a mutation sequence} in correspondence to the picture group element $\prod_{k}x_{i_kj_k}^{\delta_k}$ acting on a $c$-matrix with associated permutation $\sigma$ is defined as $\rho(\prod_{k}x_{i_kj_k}^{\delta_k}) = \sigma(\prod_{k}(i_k+1,j_k))^{-1}\sigma^{-1}$. Here $\delta_k\in\{+,-\}$.
\end{definition}

\indent In particular any mutation at a vertex with $c$-vector $\pm\beta_{ij}$ has an associated permutation $\sigma(i+1,j)\sigma^{-1}$ with $\sigma$ the permutation of the $c$-matrix before the mutation. In the special case when $i+1=j$ which is when $\beta_{ij}$ is a simple root the associated permutation is trivial.\\ 
\bibliographystyle{amsplain}

\end{document}